\documentclass[a4paper,11pt]{amsart}
\usepackage[
  top=3cm,
  bottom=2cm,
  left=1.5cm,
  right=1.5cm,
  headheight=14pt,
  headsep=1cm,
  footskip=1cm
]{geometry}
\usepackage{amsmath, epsfig, verbatim}
\usepackage{amsfonts}
\usepackage{amsthm}
\usepackage{mathtools}
\usepackage{amsmath}
\usepackage{amssymb} 
\usepackage{graphics,graphicx}
\usepackage{epstopdf}
\usepackage{listings}
\usepackage{xcolor}
\usepackage{tikz}
\usepackage{subcaption}
\usepackage{pgfplots}
\usepackage{caption}
\usepackage{booktabs}
\usepackage{tabularx}
\usepackage{hyperref}   
\usepackage{doi}
\usepackage{enumitem}
\usepackage[numbers,sort&compress]{natbib}
\usetikzlibrary{calc}
\allowdisplaybreaks
\definecolor{codegray}{rgb}{0.95,0.95,0.95}
\definecolor{pykeyword}{rgb}{0.13,0.13,1}
\definecolor{pystring}{rgb}{0.58,0,0.82}

\lstdefinestyle{pythonstyle}{
    backgroundcolor=\color{codegray},
    language=Python,
    basicstyle=\ttfamily\small,
    keywordstyle=\color{pykeyword}\bfseries,
    stringstyle=\color{pystring},
    commentstyle=\color{gray},
    showstringspaces=false,
    numbers=left,
    numberstyle=\tiny,
    frame=single,
    breaklines=true,
    tabsize=4,
}

\numberwithin{equation}{section}
\theoremstyle{plain}
\newtheorem{theorem}{Theorem}
\newtheorem{defn}[theorem]{Definition}
\newtheorem{coro}[theorem]{Corollary}
\newtheorem{lemma}[theorem]{Lemma}
\newtheorem{prop}[theorem]{Proposition}
\newtheorem{remark}[theorem]{Remark}

\begin{document}
\title{Mutual visibility in Moore graphs and $(d,2)$-graphs with defect} 
\author{Tonny K B}
\address{Tonny K B, Department of Mathematics, College of Engineering Trivandrum, Thiruvananthapuram, Kerala, India, 695016.}
\email{tonnykbd@cet.ac.in}
\author{Shikhi M}
\address{Shikhi M, Department of Mathematics, College of Engineering Trivandrum, Thiruvananthapuram, Kerala, India, 695016.}
\email{shikhim@cet.ac.in}
\begin{abstract}
The concept of mutual visibility in a graph encodes combinatorial information about vertex subsets with prescribed visibility properties and serves as a useful algebraic invariant. In this paper, we derive algebraic conditions for the mutual-visibility number of $(d,2)$-graphs with non-negative defect. We then determine this parameter for $(d,2,-2)$-graphs for $d=3$ and $4$, and establish an upper bound for $d=5$. In the case $\delta=0$, that is, for Moore graphs of diameter $2$, we focus on the
Hoffman--Singleton graph. We establish an upper bound of $20$ for its mutual-visibility number and subsequently employ an integer programming approach to show that this bound is tight. As a corollary, we deduce that the maximum size of an induced matching in the Hoffman--Singleton graph is $10$.
\end{abstract}
\maketitle
\section{Introduction}
 Let $G(V,E)$ be a simple graph and let $X\subseteq V$.  Two vertices $u,v\in V$ are said to be $X$-visible \cite{MV_3} if there exists a shortest path $P$ from $u$ to $v$ such that the internal vertices of $P$ do not belong to $X$; that is, $V(P)\cap X \subseteq \{u, v\}$. A set $X$ is called a mutual-visibility set \cite{Stefano} of $G$ if every pair of vertices in $X$ is $X$-visible. The maximum size of such a set in $G$ is referred to as the mutual-visibility number, denoted by $\mu(G)$. The visibility  polynomial of $G$, denoted by $\mathcal{V}(G)$, is defined in \cite{sandi} as $
   \mathcal{V}(G) = \sum_{i \geq 0} r_i x^{i}$, 
where $r_i$ denotes the number of mutual-visibility sets of $G$ having cardinality $i$.

The notion of mutual-visibility in graphs has received increasing attention because of its significance across both theoretical and practical domains. Wu and Rosenfeld first explored the foundational visibility problems in the context of pebble graphs~\cite{Geo_convex_1, Geo_convex_2}. Later, Di Stefano introduced the formal definition of mutual-visibility sets in graph-theoretic terms~\cite{Stefano}. The concept of mutual-visibility serves as a powerful tool for analyzing the transmission of information, influence, or coordination under topological constraints. This graph-theoretic framework has been the subject of numerous investigations~\cite{MV_1, MV_2, MV_5, MV_6, MV_9}, and several variants of mutual visibility have also been proposed in ~\cite{MV_10}. 

In practical applications, mutual visibility is significant in robotics. In multi-agent systems, robots must often reposition themselves to ensure unobstructed visibility between all agents, facilitating decentralized control algorithms for formation, navigation, and surveillance in unknown or dynamic environments. For detailed studies, see~\cite{robotics2,robotics4,robotics5,robotics7}.
 
This paper examines the mutual-visibility number of $(d,2)$-graphs with non-negative defect. In Section~\ref{P6.sec3}, relevant algebraic tools are developed to analyse the size of mutual-visibility sets in such graphs. The mutual-visibility number is determined for $(d,2,-2)$-graphs for $d=3$ and $4$, and an upper bound is established for $d=5$.  In Section~\ref{P6.sec4}, the case $\delta=0$, corresponding to Moore graphs of diameter $2$, is considered, with particular emphasis on the Hoffman--Singleton graph. An upper bound of $20$ for its mutual-visibility number is established, and an integer programming approach is used to show that this bound is tight. As a corollary, it is deduced that the maximum size of an induced matching in the Hoffman--Singleton graph is $10$.  These results illustrate how algebraic and optimisation techniques can be combined to determine precise visibility parameters in graph families. 
 
\section{Notations and preliminaries}
 In this paper, $G(V,E)$ represents an undirected simple connected graph with vertex set $V$ and edge set $E$. A sequence of distinct vertices $(u_0,u_1,u_2,\ldots,u_{n})$ is referred to as a $(u_0,u_n)$\emph{-path} in a graph $G$ if $u_iu_{i+1}\in E(G)$, $\forall i\in\{0,1,\ldots,(n-1)\}$. A \emph{cycle} in a graph $G$ is a $(u_0,u_n)$-path together with an edge $u_0u_n$. If the graph $G$ itself is a cycle, it is denoted by $C_n$. A graph $G$ is \emph{triangle-free}
if there does not exist any cycle of length $3$ in $G$.  \emph{Mantel's theorem} states that a simple graph with $n$ vertices that contains no triangles has at most $
\left\lfloor n^2/4 \right\rfloor$ edges.

Let $X$ be a set. The cardinality of $X$ is denoted by $|X|$. For a vertex $w\in V(G)$, let $N(w)$ denote the set of vertices adjacent to $w$, and let $d_G(w)=|N(w)|$ denote the degree of $w$ in $G$. We denote the set of vertices at distance $2$ from $w$ by $N_2(w)$. The maximum degree of a vertex in $G$ is denoted by $\Delta(G)$, or simply by $\Delta$ when no ambiguity arises. A graph $G$ is said to be \emph{$d$-regular} if every vertex of $G$ has degree $d$. A $d$-regular graph $G$ on $n$ vertices is said to be \emph{strongly regular} if every pair of adjacent vertices has exactly $\lambda$ common neighbours and every pair of non-adjacent vertices has exactly $\mu$ common neighbours. Such a graph is denoted by $\operatorname{srg}(n,d,\lambda,\mu)$.

A graph $G$ with maximum degree $d$ and diameter $k$ is called $(d,k)$\emph{-graph}. The maximum possible order of such a graph is $M(d,k)=1+d\sum_{i=0}^{k-1}(d-1)^i$,
which is called the \emph{Moore bound}. A graph attaining this bound is called a \emph{Moore graph}, and such graphs are necessarily $d$-regular and contain no cycles of length less than five ~\cite{hoffman_moore}. Moore graphs are extremely rare. For $k=1$ and $d\ge 1$, they are precisely the complete graphs on $d+1$ vertices. For $k=2$ and $d\ge 1$, Moore graphs exist only for $d=2,3,7$ and possibly $57$~\cite{hoffman_moore}. For $k\ge 3$ and $d=2$, they are precisely the cycles of length $2k+1$. Furthermore, for $k\ge 3$ and $d\ge 3$, Moore graphs do not exist; see~\cite{BanIto1973, Dam1973}. A detailed account of this topic can be found in~\cite{ MillerSiran2013}. Since Moore graphs are rare, it is natural to study graphs whose order is close to the Moore bound. Let $G$ be a $(d,k)$-graph of order $n$. Then the \emph{defect} of $G$ is defined by $\delta = M(d,k) - n \geq 0$, and $G$ is called a \emph{$(d,k,-\delta)$-graph}. If $\delta < 1+(d-1)+(d-1)^2+\cdots+(d-1)^{k-1}$ then the graph is regular \cite{FerMilVil2011}. Regular $(d,2,-\delta)$-graphs having $n$ vertices satisfy the matrix equation $
A^{2}+A-(d-1)I=J+B$,
where $A$ denotes the adjacency matrix of the graph, $J$ the $n\times n$ all-ones matrix, $I$ the identity matrix, and $B$ is a matrix whose row and column sums are equal to $\delta$ \cite{MillerSiran2013, ConGim09}.
  
A \emph{matching} in a graph $G(V,E)$ is a set of edges $M \subseteq E$ such that no two edges in $M$ share a common endpoint. An \emph{induced matching} \cite{induced_match_1} in $G$ is a matching $M$ in $G$ such that no two edges of $M$ are joined by an edge of $G$. If an induced matching consists of exactly $k$ edges, then it is referred to as an \emph{induced $k$-matching}. For a subset $S \subseteq V$, $G[S]$ denotes the induced subgraph of $S$ in $G$. Let $S,T\subseteq V$. Then $e(S)$ denotes the number of edges of $G[S]$ and $e(S,T)$ denotes the number of edges of $G$ having one end in $S$ and the other end in $T$.

A function $f : I \subseteq \mathbb{R} \to \mathbb{R}$ defined on an interval $I$ is said to be \emph{convex} if, for all $a, b \in I$ and $\lambda \in [0,1]$, 
$f(\lambda a + (1-\lambda)b) \leq \lambda f(a) + (1-\lambda)f(b)$; it is said to be \emph{concave} if the inequality is reversed.

\emph{Jensen's inequality }(see \cite{Cvetkovski2012}): Let $f:\,I\to\mathbb{R}$ be a convex function on an interval $I\subseteq\mathbb{R}$. 
For any integer $m\ge 1$ and any points $x_1, x_2, \dots, x_m\in I$ and nonnegative weights
$\alpha_1, \alpha_2, \dots,\alpha_m$ with $\sum_{i=1}^m\alpha_i=1$, we have
\begin{equation}\label{P6.eq3}
f\Biggl(\sum_{i=1}^m \alpha_i x_i\Biggr)
\le \sum_{i=1}^m \alpha_i f(x_i)  
\end{equation}
In particular, for equal weights $\alpha_i=\tfrac1m$,
\[
f\!\Bigl(\tfrac{1}{m}\sum_{i=1}^m x_i\Bigr)
\le \tfrac{1}{m}\sum_{i=1}^m f(x_i).
\]
If strict inequality holds in \eqref{P6.eq3}, then we say that $f$ is strictly convex. If $f$ is strictly convex and $m \geq 2$, then equality holds in Jensen's inequality if and only if $x_1=x_2=\cdots=x_m$.

\section{Mutual-visibility in $(d,2,-\delta)$-graphs}\label{P6.sec3}
In this section, we study mutual-visibility sets in $(d,2,-\delta)$-graphs, where $\delta\ge 0$ denotes the defect. We first establish general bounds and structural properties of such sets, and then specialize to the case $\delta=0$ to recover and refine results for Moore graphs of diameter $2$. We also determine the mutual-visibility number for three of the four known graphs with defect $2$ and derive an upper bound for the remaining one.
\begin{lemma}\label{P6.lem1}
Let $G$ be a graph of diameter $2$. A subset $S \subseteq V(G)$ is a mutual-visibility set of $G$ if and only if, for each non-adjacent pair $\{u,v\}\subseteq S$, there is a vertex $w\in V(G)\setminus S$ that is adjacent to both $u$ and $v$.
\end{lemma}
\begin{proof}
Let $S \subseteq V(G)$. If $u,v \in S$ are non-adjacent, then $d_G(u,v)=2$. Thus, every shortest $(u,v)$-path has exactly one internal vertex, which is a common neighbour of $u$ and $v$. Consequently, $u$ and $v$ are $S$-visible if and only if at least one of their common neighbours lies in $V(G)\setminus S$. Hence $S$ is a mutual-visibility set of $G$ if and only if for every non-adjacent pair $\{u,v\}\subseteq S$ there exists a common neighbour $w \notin S$.
\end{proof}

\begin{lemma}\label{P6.lem2}
Let $G$ be a $(d,2,-\delta)$-graph, and let $u\in V(G)$. Then
\[
\sum_{v\in N_2(u)}\bigl(|N(u)\cap N(v)|-1\bigr)
\le
\delta-d\bigl(d-d_G(u)\bigr)-2e\bigl(N(u)\bigr).
\]
Moreover, if $G$ is regular, then
\begin{equation}\label{P6.eq5}
\sum_{v\in N_2(u)}\bigl(|N(u)\cap N(v)|-1\bigr)
=
\delta-2e\bigl(N(u)\bigr).   
\end{equation}
\end{lemma}

\begin{proof}
Let $u \in V(G)$. Since $G$ has diameter $2$, every vertex outside
$N[u]$ lies at distance exactly 2 from $u$. Hence $|N_2(u)|=|V(G)|-1-d_G(u)=d^2-\delta-d_G(u)$.

We count the edges between $N(u)$ and $N_2(u)$ in two ways. On the one hand,
for each $v\in N_2(u)$, the number of neighbours of $v$ in $N(u)$ is
$|N(u)\cap N(v)|$, and hence
\[
e\bigl(N(u),N_2(u)\bigr)=\sum_{v\in N_2(u)} |N(u)\cap N(v)|.
\]
On the other hand, for each $x\in N(u)$, among the $d_G(x)-1$ neighbours of
$x$ other than $u$, exactly $d_{G[N(u)]}(x)$ lie in $N(u)$, and the remaining
ones lie in $N_2(u)$. Therefore
\[
e\bigl(N(u),N_2(u)\bigr) = \sum_{x\in N(u)}\left[ (d_G(x)-1)- d_{G[N(u)]}(x) \right]
=
\left[\sum_{x\in N(u)}(d_G(x)-1)\right]-2e\bigl(N(u)\bigr) \leq d_G(u)(d-1)-2e\bigl(N(u)\bigr).
\]
Thus
\[
\sum_{v\in N_2(u)}\bigl(|N(u)\cap N(v)|-1\bigr)
=
e\bigl(N(u),N_2(u)\bigr)-|N_2(u)|
\]
\[
\le d_G(u)(d-1)-2e\bigl(N(u)\bigr)-(d^2-\delta-d_G(u))
=
\delta-d(d-d_G(u))-2e\bigl(N(u)\bigr),
\]
which proves the first inequality. If $G$ is $d$-regular, then
\[
\sum_{x\in N(u)}(d_G(x)-1)=d(d-1)
\quad\text{and}\quad
|N_2(u)|=d^2-\delta-d=d(d-1)-\delta,
\]
so equality holds, giving $
\sum_{v\in N_2(u)}\bigl(|N(u)\cap N(v)|-1\bigr)
=\delta-2e\bigl(N(u)\bigr)$. This completes the proof.
\end{proof}

The quantity $\sum_{v\in N_2(u)}\bigl(|N(u)\cap N(v)|-1\bigr)$ measures the total excess in the number of shortest $(u,v)$-paths of length $2$, over all vertices $v$ at distance $2$ from $u$. Thus, the identity \eqref{P6.eq5} shows that the global defect $\delta$ is exactly distributed between the local clustering at $u$, captured by $2e(N(u))$, and the excess multiplicity of shortest $(u,v)$-paths for $v\in N_2(u)$. In particular, when $\delta=2$, the identity \eqref{P6.eq5} imposes strong local restrictions. Since the left-hand side is non-negative, we obtain $e\bigl(N(u)\bigr)\le 1$ for every $u\in V(G)$, so the neighbourhood of any vertex contains at most one edge. If $e\bigl(N(u)\bigr)=1$, then the left-hand side vanishes, and hence $|N(u)\cap N(v)|=1$ for all $v\in N_2(u)$, implying that $u$ lies in no $4$-cycle. If $e\bigl(N(u)\bigr)=0$, then the left-hand side equals $2$, and thus either there is a unique vertex $v\in N_2(u)$ with $|N(u)\cap N(v)|=3$, or there are exactly two vertices $v_1,v_2\in N_2(u)$ with $|N(u)\cap N(v_i)|=2, \ i=1,2$. In particular, $u$ lies in exactly three or two distinct $4$-cycles, respectively. These constraints severely restrict the local structure of $(d,2,-2)$-graphs and are useful in ruling out potential configurations by inspecting the neighbourhood of each vertex.

The identity \eqref{P6.eq5} can also be viewed as a consequence of the matrix equation
$A^{2}+A-(d-1)I=J+B$. Indeed, for distinct vertices $u$ and $v$, the $(u,v)$-entry of this equation gives $B_{uv}=|N(u)\cap N(v)|+A_{uv}-1$.
In particular, if $v\in N(u)$, then $B_{uv}=|N(u)\cap N(v)|$, while if $v\in N_2(u)$, then $B_{uv}=|N(u)\cap N(v)|-1$. Summing the $u$-th row of $B$ over $N(u)$ and $N_2(u)$ yields
\[
\sum_{v\in N(u)}B_{uv}=\sum_{v\in N(u)} d_{G[N(u)]}(v)=2e(N(u))
\quad\text{and}\quad
\sum_{v\in N_2(u)}B_{uv}
=
\sum_{v\in N_2(u)}\bigl(|N(u)\cap N(v)|-1\bigr).
\]
Since each row sum of $B$ is equal to $\delta$ and $B_{uu}=0$, the identity \eqref{P6.eq5} follows by decomposing the $u$-th row sum of $B$ into these two contributions.
\begin{coro}\label{P6.coro5}
Let $G$ be a $(d,2,-\delta)$-graph. If $G$ is triangle-free, then
\[
\sum_{v\in N_2(u)}\bigl(|N(u)\cap N(v)|-1\bigr)
\le
\delta-d\bigl(d-d_G(u)\bigr).
\]
If, in addition, $G$ is regular, then 
\[\sum_{v\in N_2(u)}\bigl(|N(u)\cap N(v)|-1\bigr)
=
\delta.\]
\end{coro}

\begin{lemma}\label{P6.lem3}
Let $G$ be a graph of diameter $2$, and let $S \subseteq V(G)$ be a
mutual-visibility set of size $s$. Then:
\begin{enumerate}
    \item[(i)] $ \displaystyle
    \binom{s}{2} - e(S) \leq \sum_{w \in V(G)\setminus S}
    \binom{|N(w)\cap S|}{2}$, 
    where $e(S)$ denotes the number of edges in the subgraph $G[S]$.

    \item[(ii)] Moreover, if $G$ is a regular $(d,2,-\delta)$-graph, then
    \[
    \sum_{w \in V(G)\setminus S}\binom{|N(w)\cap S|}{2}
    \le
    \binom{s}{2}-e(S)
    +\frac12\sum_{u\in S}\bigl(\delta-2e(N(u))\bigr)
    +\sum_{\{u,v\}\in\mathcal A} m(u,v),
    \]
    where $\mathcal A=\bigl\{\{u,v\}\subseteq S:\ uv\in E(G)\bigr\}$,
    and for each $\{u,v\}\in\mathcal A$, $m(u,v)$ denotes the number of common
    neighbours of $u$ and $v$ that lie in $V(G)\setminus S$.
\end{enumerate}
\end{lemma}
 \begin{proof}
 Let $\mathcal{U} = \bigl\{\{u,v\}\subseteq S: uv\notin E(G)\bigr\}$ be the set of unordered non-adjacent pairs of vertices in $S$. Then $|\mathcal{U}| = \binom{s}{2}-e(S)$. By Lemma~\ref{P6.lem1},  for each $\{u,v\}\in\mathcal{U}$ there exists at least one vertex $w \in V(G)\setminus S$ such that $u,v\in N(w)$. For each $w\in V(G)\setminus S$, define $\mathcal{F}(w) = \bigl\{\{u,v\}\subseteq S : u,v \in N(w)\bigr\}$.
Then, $|\mathcal{F}(w)|=\binom{|N(w)\cap S|}{2}$.  Therefore
\[
|\mathcal{U}|  \leq \sum_{w\in V(G)\setminus S} |\mathcal{F}(w)|
=  \sum_{w\in V(G)\setminus S} \binom{|N(w)\cap S|}{2}.
\]
Since $|\mathcal{U}|=\binom{s}{2}-e(S)$, the desired inequality follows,
which proves~(i).

\smallskip
Now assume that $G$ is a regular $(d,2,-\delta)$-graph. 
Let $\mathcal{A}=\bigl\{\{u,v\}\subseteq S:\ uv\in E(G)\bigr\}$. For each pair $\{u,v\}\subseteq S$, let $m(u,v)$ denote the number of common neighbours of $u$ and $v$ in $V(G)\setminus S$. Then,
\begin{equation}\label{P6.eq10}
  \sum_{w\in V(G)\setminus S}\binom{|N(w)\cap S|}{2}
=
\sum_{\{u,v\}\in\mathcal{U}} m(u,v)
+
\sum_{\{u,v\}\in\mathcal{A}} m(u,v)=|\mathcal{U}|+\sum_{\{u,v\}\in\mathcal{U}}(m(u,v)-1)+\sum_{\{u,v\}\in\mathcal{A}} m(u,v)  
\end{equation}
Note that both sides of equation~\eqref{P6.eq10} count the number of triples $(u,v,w)$ such that
$\{u,v\}\subseteq S$, $w\in V(G)\setminus S$, and $w$ is a common neighbour of
$u$ and $v$. Fix a vertex $u\in S$. Since $G$ has diameter $2$, every vertex $v\in S$ with
$uv\notin E(G)$ lies in $N_2(u)$. Also, every common neighbour of $u$ and $v$
lying in $V(G)\setminus S$ is in particular a common neighbour of $u$ and $v$
in $G$. Hence $m(u,v)\le |N(u)\cap N(v)|$. Thus
\[
\sum_{\substack{v\in S\\ uv\notin E(G)}}(m(u,v)-1)
\le
\sum_{\substack{v\in S\\ uv\notin E(G)}}
\bigl(|N(u)\cap N(v)|-1\bigr)
\le
\sum_{v\in N_2(u)}\bigl(|N(u)\cap N(v)|-1\bigr).
\]
Lemma~\ref{P6.lem2} gives
\[
\sum_{\substack{v\in S\\ uv\notin E(G)}}(m(u,v)-1)
\le
\delta-2e(N(u)).
\]

Summing over all $u\in S$, we obtain
\[
\sum_{u\in S}\sum_{\substack{v\in S\\ uv\notin E(G)}}(m(u,v)-1)
\le
\sum_{u\in S}\bigl(\delta-2e(N(u))\bigr).
\]
Each unordered pair $\{u,v\}\in\mathcal{U}$ is counted exactly twice on the
left-hand side, and hence
\[
2\sum_{\{u,v\}\in\mathcal{U}}(m(u,v)-1)
\le
\sum_{u\in S}\bigl(\delta-2e(N(u))\bigr).
\]
Therefore
\[
\sum_{\{u,v\}\in\mathcal{U}}(m(u,v)-1)
\le
\frac12\sum_{u\in S}\bigl(\delta-2e(N(u))\bigr).
\]

From \eqref{P6.eq10}, it follows that
\begin{equation}\label{P6.eq7}
\sum_{w\in V(G)\setminus S}\binom{|N(w)\cap S|}{2}
\leq
\binom{s}{2}-e(S)
+\frac12\sum_{u\in S}\bigl(\delta-2e(N(u))\bigr)
+\sum_{\{u,v\}\in\mathcal{A}} m(u,v)
\end{equation}
\end{proof}   
\begin{coro}\label{P6.coro3}
Under the hypotheses of Lemma~\ref{P6.lem3}(ii), if $G$ is triangle-free, then
\[
\sum_{w \in V(G)\setminus S}\binom{|N(w)\cap S|}{2}
\le
\binom{s}{2}-e(S)+\frac{\delta s}{2}.
\]
If, in addition, $\delta=0$, then
\begin{equation}\label{P6.eq8}
\binom{s}{2}-e(S)
=
\sum_{w \in V(G)\setminus S}\binom{|N(w)\cap S|}{2}.    
\end{equation}
\end{coro}

\begin{proof}
If $G$ is triangle-free, then $e(N(u))=0$ for every $u\in S$, and no adjacent
pair of vertices has a common neighbour. Hence $m(u,v)=0$ for every $\{u,v\}\in\mathcal A$. Therefore, Lemma~\ref{P6.lem3}(ii) reduces to
\[
\sum_{w \in V(G)\setminus S}\binom{|N(w)\cap S|}{2}
\le
\binom{s}{2}-e(S)+\frac12\sum_{u\in S}\delta
=
\binom{s}{2}-e(S)+\frac{\delta s}{2}.
\]
If $\delta=0$, the last term vanishes, and part (i) of Lemma~\ref{P6.lem3}  gives the reverse
inequality. Hence equality holds.

\end{proof}
\begin{coro}\label{P6.coro4}
Under the hypotheses of Lemma~\ref{P6.lem3}(i), if $G$ is triangle-free and every non-adjacent pair $\{u,v\} \subseteq S$ has a unique common neighbour in $V(G)\setminus S$, then \eqref{P6.eq8} holds.
\end{coro}
\begin{proof}
Since $G$ is triangle-free, we have $m(u,v)=0$ for every $\{u,v\}\in\mathcal A$.
Moreover, every non-adjacent pair $\{u,v\} \subseteq S$ has a unique common
neighbour in $V(G)\setminus S$, and hence $m(u,v)=1$ for every
$\{u,v\}\in\mathcal U$. Thus, by \eqref{P6.eq7},
\[
\sum_{w\in V(G)\setminus S}\binom{|N(w)\cap S|}{2}
=
|\mathcal U|
+\sum_{\{u,v\}\in\mathcal U}(m(u,v)-1)
+\sum_{\{u,v\}\in\mathcal A}m(u,v)
=
|\mathcal U|
=
\binom{s}{2}-e(S).
\]
This completes the proof.
\end{proof}
\begin{lemma}\label{P6.lem7}
Let $G$ be a graph, let $S\subseteq V(G)$, and put $T=V(G)\setminus S$. Then
\[
e(S,T) \;=\; \sum_{v\in S}\deg_G(v)-2e(S)
\;=\; \sum_{v\in S}\bigl(\deg_G(v)-\deg_{G[S]}(v)\bigr).
\]
In particular, if $G$ is $d$-regular, then $e(S,T)=d|S|-2e(S)$.
\end{lemma}
\begin{proof}
For every $v\in S$, let $t(v)$ count the edges joining $v$ to the vertices in $T$. Then for each vertex $v$ belonging to $S$, $\deg_G(v)=\deg_{G[S]}(v)+t(v)$. Summing this identity over all \(v\in S\) yields
\[
\sum_{v\in S}\deg_G(v)=\sum_{v\in S}\deg_{G[S]}(v)+\sum_{v\in S}t(v).
\]
By the Handshaking Lemma applied to the induced subgraph \(G[S]\), we have $\sum_{v\in S}\deg_{G[S]}(v)=2e(S)$.
Moreover, the sum \(\sum_{v\in S}t(v)\) counts each edge with one endpoint in \(S\) and other in \(T\) exactly once, which implies that
\[
\sum_{v\in S}t(v)=e(S,T).
\]
Therefore, $
\sum_{v\in S}\deg_G(v)=2e(S)+e(S,T)$, 
and rearranging yields
\[
e(S,T)=\sum_{v\in S}\deg_G(v)-2e(S)=\sum_{v\in S}\bigl(\deg_G(v)-\deg_{G[S]}(v)\bigr).
\]
If \(G\) is \(d\)-regular then \(\sum_{v\in S}\deg_G(v)=d|S|\) and hence \(e(S,T)=d|S|-2e(S)\).
\end{proof}
\begin{theorem}\label{P6.th1} Let $G$ be a $(d,2)$-graph of order $n\geq 2$, where $d\geq 2$. If $S \subseteq V(G)$ is a mutual-visibility set of size $s$, then $ s^{2} + (d^{2}-2d-1)s - d(d-1)n \leq 0$,
and consequently
\[
0 \leq s \leq
\left\lfloor
\frac{ -(d^{2}-2d-1)
+ \sqrt{(d^{2}-2d-1)^{2} + 4d(d-1)n} }{2}
\right\rfloor.
\]
\end{theorem}
\begin{proof}
Since $0 \le |N(w)\cap S| \le d_G(w) \le d$, we have $
\binom{|N(w)\cap S|}{2} \le \binom{d}{2}\quad\text{for every } w\notin S$. Summing over the $n-s$ vertices outside $S$ yields
\begin{equation}\label{P6.eq9}
\sum_{w\notin S}\binom{|N(w)\cap S|}{2}
\le \sum_{w\notin S}\binom{d}{2}
= (n-s)\binom{d}{2}.    
\end{equation}
By Lemma~\ref{P6.lem3}(i) and the inequality above,
\[
\binom{s}{2} - e(S) \le \sum_{w\notin S}\binom{|N(w)\cap S|}{2}
\le (n-s)\binom{d}{2}
\implies
\binom{s}{2} \le (n-s)\binom{d}{2} + e(S).
\]
By the handshaking lemma, $2e(S)=\sum_{v\in S} d_{G[S]}(v) \le \sum_{v\in S} d_G(v) \le sd$, so that $e(S)\le sd/2$. Therefore,
$ \frac{s(s-1)}{2} \le (n-s)\frac{d(d-1)}{2} + \frac{sd}{2}$, so that $s^2 + (d^2 - 2d - 1)\,s - d(d-1)\,n \le 0$.

Since the quadratic $f(s)=s^2+(d^2-2d-1)s-d(d-1)n$ has positive leading coefficient, the inequality $f(s)\le 0$ holds precisely on the closed interval $[r^-,r^+]$, where $r^-$ and $r^+$ are the two real roots of $f(s)$. These roots exist because the discriminant of $f(s)$ is $
D = (d^2 - 2d - 1)^2 + 4d(d-1)n,
$
and both terms, $(d^2 - 2d - 1)^2$ and $4d(d-1)n$, are strictly positive for all $d\geq 2$ and $n\geq 2$. Indeed, the polynomial $d^2 - 2d - 1$ has roots $1 \pm \sqrt{2}$, which are not integers, ensuring that $(d^2 - 2d - 1)^2 > 0$ for all integers $d$. Thus $r^{\pm}=\frac{-(d^2-2d-1)\pm \sqrt{D}}{2}$. Moreover, $f(0)=-d(d-1)n<0$, since $d(d-1)n>0$ for $d,n \geq 2$. Therefore $r^-<0<r^+$.  
As $s\ge 0$ and $s$ is an integer, it follows that
\[
0 \le s \le \lfloor r^+\rfloor
=\left\lfloor 
\frac{ -(d^2-2d-1) 
+ \sqrt{(d^2-2d-1)^2+4d(d-1)n} }{2}
\right\rfloor.
\]
\end{proof}
The preceding proof can be sharpened in the regular case by replacing the
estimate \(e(S)\le sd/2\) with the exact relation
\(e(S,T)=ds-2e(S)\), where $ T=V(G)\setminus S$.

\begin{theorem}\label{P6.coro6}
Let $G$ be a regular $(d,2)$-graph of order $n\ge 2$ and let
$S\subseteq V(G)$ be a mutual-visibility set of size $s$. Let $T=V(G)\setminus S$.
Then $s^{2}+(d^{2}-2d-1)s-d(d-1)n+e(S,T)\le 0$. Moreover, if $n-s\le d+1$, then $(d^{2}-3d+2n-2)s-n(d^{2}-2d+n-1)\le 0$.
\end{theorem}

\begin{proof}
As in the proof of Theorem~\ref{P6.th1}, Lemma~\ref{P6.lem3}(i) and
\eqref{P6.eq9} imply that $\binom{s}{2}\le (n-s)\binom{d}{2}+e(S)$. Since $G$
is $d$-regular, Lemma~\ref{P6.lem7} gives $e(S,T)=ds-2e(S)$, and hence
$e(S)=\frac{ds-e(S,T)}{2}$. Substituting this into the previous inequality and
simplifying yields 
\begin{equation}\label{P6.eq6}
 s^{2}+(d^{2}-2d-1)s-d(d-1)n+e(S,T)\le 0   
\end{equation}

By Lemma~\ref{P6.lem7}, $e(S,T)=d|T|-2e(T)\ge d|T|-2\binom{|T|}{2}=|T|(d-|T|+1)$. Note that this lower bound is non-negative since $n-s \leq d+1$.
Substituting this lower bound for $e(S,T)$ into the inequality \eqref{P6.eq6}, we get
$s^{2}+(d^{2}-2d-1)s-d(d-1)n+|T|(d-|T|+1)\le 0$. 
Substituting $|T|=n-s$ and simplifying, we obtain the result.
\end{proof}

Next, we derive bounds for $\mu(G)$ when $G$ is a triangle-free $(d,2,-\delta)$-graph with $\delta\ge 0$.
\begin{lemma}\label{P6.lem4}
Let $G$ be a triangle-free $(d,2,-\delta)$-graph, and let $S\subseteq V(G)$.
Then the following hold:
\begin{enumerate}
    \item[(i)] If $\Delta(G[S])\le 1$, then $S$ is a mutual-visibility set of $G$.
    \item[(ii)] If $S$ is a mutual-visibility set of $G$, then $\Delta(G[S])\le \delta+1$.
    \item[(iii)] If $S$ is a mutual-visibility set of $G$, then $0\le 2e(S)\le (\delta+1)|S|$.
\end{enumerate}
\end{lemma}

\begin{proof}
(i) Assume that $\Delta(G[S])\le 1$. Let $u,v\in S$ be non-adjacent. Since $G$ has diameter $2$, the vertices $u$ and $v$ have a common neighbour, say $w$. If $w\in S$, then $w$ is adjacent to both $u$ and $v$ in $G[S]$, and hence $d_{G[S]}(w)\ge 2$, contradicting $\Delta(G[S])\le 1$. Therefore $w\notin S$, and so, by Lemma~\ref{P6.lem1}, $S$ is a mutual-visibility set of $G$.

(ii) Assume that $S$ is a mutual-visibility set of $G$. Let $x\in S$, and set
$k=d_{G[S]}(x)$. We show that $k\le \delta+1$. If $k\le 1$, the claim is clear.
So assume that $k\ge 2$, and let $N_{G[S]}(x)=\{u_1,u_2,\dots,u_k\}$. Since $G$
is triangle-free, the vertices $u_1,u_2,\dots,u_k$ are pairwise non-adjacent. Fix $u_1$. For each $i\in\{2,\dots,k\}$, the vertices $u_1$ and $u_i$ are
non-adjacent vertices of $S$, and $x\in S$ is a common neighbour of $u_1$ and
$u_i$. Since $S$ is a mutual-visibility set, Lemma~\ref{P6.lem1} implies that
$u_1$ and $u_i$ also have a common neighbour in $V(G)\setminus S$. Hence
$|N(u_1)\cap N(u_i)|\ge 2$ for each $i=2,\dots,k$.

Moreover, since $G$ has diameter $2$, we have $u_i\in N_2(u_1)$ for each
$i=2,\dots,k$. Therefore each of the vertices $u_2,\dots,u_k$ contributes at
least $1$ to the sum $\sum_{v\in N_2(u_1)}\bigl(|N(u_1)\cap N(v)|-1\bigr)$, and hence
\[
k-1\le
\sum_{v\in N_2(u_1)}\bigl(|N(u_1)\cap N(v)|-1\bigr).
\]

By Corollary~\ref{P6.coro5},
\[
\sum_{v\in N_2(u_1)}\bigl(|N(u_1)\cap N(v)|-1\bigr)
\le \delta - d\bigl(d-d_G(u_1)\bigr)\le \delta.
\]
Thus $k-1\le \delta$, and hence $k\le \delta+1$. Since $x\in S$ was arbitrary,
it follows that $\Delta(G[S])\le \delta+1$.

(iii) Clearly, $e(S)\ge 0$. By part (ii), $\Delta(G[S])\le \delta+1$, and hence $\sum_{v\in S} d_{G[S]}(v)\le (\delta+1)|S|$. By the handshaking lemma, $2e(S)=\sum_{v\in S} d_{G[S]}(v)$, and the result follows.
\end{proof}
\begin{coro}\label{P6.coro2}
Let $G$ be a Moore graph of diameter 2. If $S\subseteq V(G)$ is a mutual-visibility set, then $
0\leq 2e(S)\le |S|$, where $e(S)$ denotes the number of edges in the induced subgraph $G[S]$.
\end{coro}
\begin{proof}
Since $G$ is a Moore graph of diameter $2$, it is a
triangle-free $(d,2,-0)$-graph for some $d$. Thus Lemma~\ref{P6.lem4}(iii)
applies with $\delta=0$, and yields $0\le 2e(S)\leq |S|$.
\end{proof}
\begin{coro}\label{P6.coro1}
Let $G$ be a Moore graph of diameter $2$. Then the following conditions are equivalent:
\begin{enumerate}
    \item[(i)] $S$ is a mutual-visibility set of $G$.
    \item[(ii)] The induced subgraph $G[S]$ satisfies $\Delta(G[S]) \le 1$.
    \item[(iii)] $G[S]$ is a disjoint union of an induced matching and isolated vertices.
\end{enumerate}
\end{coro}
\begin{proof}
Since $G$ is a Moore graph of diameter $2$, we have $\delta=0$ and $G$ is triangle-free. Hence $(i)\iff(ii)$ follows immediately from Lemma~\ref{P6.lem4}(i) and~(ii).

\smallskip
\noindent(ii) $\Rightarrow$ (iii)  
Assume $\Delta(G[S])\leq 1$. Then every connected component of $G[S]$ has maximum degree at most $1$, and hence each component is either $K_1$ or $K_2$. Let $M$ be the set of edges corresponding to the $K_2$-components. These edges are pairwise disjoint, so $M$ is a matching.  
To see that $M$ is induced, let $v(M)$ denote the set of end vertices of edges in $M$. If there were an edge in $G$ joining vertices of $v(M)$ belonging to different $K_2$-components, then in $G[S]$ one of those endpoints would have degree at least $2$, contradicting $\Delta(G[S])\leq 1$. Thus $G[S]$ contains exactly the edges of $M$ so that $G[S]$ is a disjoint union of an induced matching and isolated vertices.

\smallskip
\noindent(iii)$\Rightarrow$ (ii)  
If $G[S]$ is a disjoint union of an induced matching and isolated vertices, then every vertex of $G[S]$ is incident to at most one edge, and hence $\Delta(G[S])\leq 1$.

\smallskip
This completes the proof of equivalence.
\end{proof}
\begin{theorem}\label{P6.th3}
Let $G$ be a triangle-free $(d,2)$-graph of order $n\ge 2$, where $d\ge 2$. If $S\subseteq V(G)$ is a mutual-visibility set with $|S|=s$, then $
s^{2}+\bigl(2d(d-1)-2\bigr)s-2d(d-1)n \leq 0$, and consequently
\[
0 \leq s \leq
\left\lfloor
\frac{ -(2d(d-1)-2)
+ \sqrt{(2d(d-1)-2)^{2} + 8d(d-1)n} }{2}
\right\rfloor.
\]
\end{theorem}

\begin{proof}
By Lemma~\ref{P6.lem3}(i) and inequality \eqref{P6.eq9}, we have $\binom{s}{2}-e(S)\le (n-s)\binom{d}{2}$. Since $G$ is triangle-free, the induced subgraph $G[S]$ is also triangle-free.
Hence, by Mantel's theorem, $e(S)\leq s^{2}/4$. Therefore,
\[
\binom{s}{2}-\frac{s^{2}}{4}\le (n-s)\binom{d}{2}
\;\Longrightarrow\;
\frac{s(s-2)}{4}\le (n-s)\frac{d(d-1)}{2}.
\]
Multiplying by $4$ and rearranging yields $s^{2}+\bigl(2d(d-1)-2\bigr)s-2d(d-1)n \le 0$.
Now consider the quadratic polynomial $
f(s)=s^{2}+\bigl(2d(d-1)-2\bigr)s-2d(d-1)n$. Since the leading coefficient is positive, the inequality $f(s)\le 0$ holds
precisely for $s$ between its two real roots. The discriminant is $ D=(2d(d-1)-2)^{2}+8d(d-1)n>0$, and the positive root is $\frac{ -(2d(d-1)-2)
+ \sqrt{D}}{2}$. Since $f(0)<0$ and $s\ge 0$ is an integer, the stated bound follows.
\end{proof}
\begin{theorem}\label{P6.th2}
Let $G$ be a triangle-free $(d,2,-\delta)$-graph, where $d\ge 2$, and let
$S\subseteq V(G)$ be a mutual-visibility set of size $s$. Then $
s^{2}+\bigl(d(d-1)-\delta-2\bigr)s-d(d-1)(d^{2}+1-\delta)\le 0$.
Consequently,
\[
0\le s\le
\left\lfloor
\frac{-(d(d-1)-\delta-2)+
\sqrt{(d(d-1)-\delta-2)^{2}+4d(d-1)(d^{2}+1-\delta)}}{2}
\right\rfloor.
\]
\end{theorem}
\begin{proof}
Since $G$ is a $(d,2,-\delta)$-graph, its order is $n=d^2+1-\delta$. Since $G$
is triangle-free, Lemma~\ref{P6.lem4}(iii) gives $2e(S)\le (\delta+1)s$. By
Lemma~\ref{P6.lem3}(i) and \eqref{P6.eq9}, we have
$\binom{s}{2}-e(S)\le (n-s)\binom{d}{2}$. Therefore
\[
\frac{s(s-1)}{2}-\frac{(\delta+1)s}{2}
\le
(d^2+1-\delta-s)\frac{d(d-1)}{2},
\]
which simplifies to $s^{2}+\bigl(d(d-1)-\delta-2\bigr)s-d(d-1)(d^{2}+1-\delta)\le 0$. The stated bound on $s$ follows from the positive root of the quadratic polynomial $f(s)=s^{2}+\bigl(d(d-1)-\delta-2\bigr)s-d(d-1)(d^{2}+1-\delta)$, together with the facts that $f(0)<0$ and $s$ is a non-negative integer.
\end{proof}
\begin{remark}
If $\delta+1 < \frac{s}{2}$, then $\frac{(\delta+1)s}{2}$ provides a better upper bound for $e(S)$ than $\frac{s^{2}}{4}$. In this case, Theorem~\ref{P6.th2} yields a stronger bound on $s$ than Theorem~\ref{P6.th3}.
\end{remark}

The only $(d,k,-1)$-graph is $C_{2k}$ for $d=2$ and $k\ge 2$~\cite{ErdSieHof1980, BanIto1981}. There are four known $(d,2,-2)$-graphs. It is known that, up to isomorphism, for $d=3$ there are exactly two such graphs of order $8$~\cite{Jorgensen1992} as shown in Figure~\ref{P6.fig1}. One is triangle-free, while the other contains a triangle. For $d=4$ and $5$, the graph is unique~\cite{Elspas1964}; see Figure~\ref{P6.fig2}. For further details on graphs with defect, see~\cite{MillerSiran2013}. We next determine the mutual-visibility number for $d=3, 4$ and derive upper bound for $d=5$.
\begin{center}
\begin{minipage}{0.4\textwidth}
\centering
\begin{tikzpicture}[
    scale=0.9,line width=0.8pt,
    every node/.style={circle, draw, fill=white, inner sep=1.2pt, font=\small}, minimum size=3.5mm,
    >=stealth
]

\node (a) at (0,1.5) {$a$};
\node (b) at (1.5,1.5) {$b$};
\node (c) at (3,1.5) {$c$};
\node (d) at (4.5,1.5) {$d$};
\node (e) at (0,0) {$e$};
\node (f) at (1.5,0) {$f$};
\node (g) at (3,0) {$g$};
\node (h) at (4.5,0) {$h$};

\draw (a)--(b)--(c)--(d);
\draw (e)--(f)--(g)--(h);

\draw (a)--(e);
\draw (b)--(f);
\draw (c)--(g);
\draw (d)--(h);

\draw (a)--(h);
\draw (e)--(d);

\end{tikzpicture}

\end{minipage}
\hfill
\begin{minipage}{0.5\textwidth}
\centering
\begin{tikzpicture}[
    scale=0.9,line width=0.8pt,
    every node/.style={circle, draw, fill=white, inner sep=1.2pt, font=\small}, minimum size=3.5mm,
    >=stealth
]

\node (a) at (0,1.6) {$l$};
\node (b) at (0,0)   {$m$};
\node (c) at (1.8,0.8) {$n$};
\node (d) at (5.0,1.6) {$o$};
\node (e) at (3.3,0.8) {$p$};
\node (f) at (5.0,0.8) {$q$};
\node (g) at (5.0,0)   {$r$};
\node (h) at (7.0,0.8) {$s$};

\draw (a)--(b)--(c)--(a);

\draw (a)--(d);
\draw (b)--(g);
\draw[bend right=15] (c) to (f);

\draw (d)--(e)--(f);
\draw (e)--(g);
\draw (d)--(h);
\draw (f)--(h);
\draw (g)--(h);

\end{tikzpicture}
\end{minipage}
\captionof{figure}{The two non-isomorphic  $(3,2,-2)$-graphs}\label{P6.fig1}
\end{center}
\begin{theorem}\label{P6.thm-defect2}
Let $G$ be a $(3,2,-2)$-graph of order $8$. Then $\mu(G)=5$.
\end{theorem}
\begin{proof}
Up to isomorphism, there are exactly two $(3,2,-2)$-graphs of order $8$, one triangle-free and the other containing a triangle (see Figure~\ref{P6.fig1}). We consider these cases separately.

\medskip
\noindent\emph{Case 1. $G$ is triangle-free.}
Let $S$ be a mutual-visibility set of $G$ with size $s$. By Theorem~\ref{P6.th2} with $d=3$ and $\delta=2$, we obtain $s^{2}+2s-48\le 0$, and hence $s\le 6$. We claim that $s\neq 6$. Suppose, to the contrary, that $s=6$, and let $T=V(G)\setminus S$. Then $|T|=2$. By Lemma~\ref{P6.lem3}(i), $\binom{6}{2}-e(S)\le \sum_{w\in T}\binom{|N(w)\cap S|}{2}\le 2\binom{3}{2}=6$, and hence $e(S)\ge 9$.

On the other hand, since $G$ is $3$-regular, by Lemma~\ref{P6.lem7}, we have $3|T|=2e(T)+e(S,T)$ and $3|S|=2e(S)+e(S,T)$. As $|T|=2$, we have $e(T)\le 1$, and hence $e(S,T)\ge 4$. Thus $18=2e(S)+e(S,T)\ge 2e(S)+4$, which implies $e(S)\le 7$, a contradiction. Therefore $s\le 5$. To see that this bound is attained, consider the set $S=\{a,b,d,e,g\}$, where the vertices are labeled as in Figure~\ref{P6.fig1}. Then every non-adjacent pair of vertices in $S$ has a common neighbour in $V(G)\setminus S$. By Lemma~\ref{P6.lem1}, $S$ is a mutual-visibility set, and hence $\mu(G)\ge 5$.

\medskip
\noindent\emph{Case 2. $G$ contains a triangle.} Assume that there exists a mutual-visibility set $S$ of size $s\ge 4$, and let $T=V(G)\setminus S$. Then $|T|=n-s\le 4$. By Theorem~\ref{P6.coro6}, it follows that $s\le 5$. To see that this bound is attained, consider the set
$S=\{l,m,n,p,s\}$, where the vertices are labeled as in
Figure~\ref{P6.fig1}. Then every non-adjacent pair of
vertices in $S$ has a common neighbour in $V(G)\setminus S$, and hence,
by Lemma~\ref{P6.lem1}, $S$ is a mutual-visibility set. Thus $\mu(G)\ge 5$.

In both cases, we have $\mu(G)\le 5$ and $\mu(G)\ge 5$, and hence
$\mu(G)=5$.
\end{proof}
\begin{center}
\begin{minipage}{0.4\textwidth}
\centering
\begin{tikzpicture}[
    scale=1.5,line width=0.8pt,
    every node/.style={circle, draw, fill=white, inner sep=1.2pt, font=\tiny}, minimum size=3.5mm
]

\coordinate (v12) at ( 0.000,  2);
\coordinate (v13) at (-2.000, -1.500);
\coordinate (v14) at ( 2.000, -1.500);

\coordinate (v0)  at (-0.95, -0.3);
\coordinate (v1)  at ( -0.5, -0.13);
\coordinate (v2)  at ( 0.36,  0.8);
\coordinate (v3)  at ( -0.27,  0.4);
\coordinate (v4)  at (-0.36,  0.8);
\coordinate (v5)  at ( 0.27, 0.4);
\coordinate (v6)  at ( 0.95, -0.3);
\coordinate (v7)  at ( 0.5,  -0.13);
\coordinate (v8)  at ( -1, -1.1);
\coordinate (v9)  at (0.3, -0.7);
\coordinate (v10) at (1, -1.1);
\coordinate (v11) at (-0.3, -0.7);

\draw (v0) -- (v1);
\draw (v0) -- (v3);
\draw (v0) -- (v11);
\draw (v0) -- (v12);

\draw (v1) -- (v4);
\draw (v1) -- (v7);
\draw (v1) -- (v8);

\draw (v2) -- (v3);
\draw (v2) -- (v5);
\draw (v2) -- (v14);

\draw (v3) -- (v4);
\draw (v3) -- (v9);

\draw (v4) -- (v5);
\draw (v10) -- (v7);
\draw (v4) -- (v13);

\draw (v5) -- (v6);
\draw (v5) -- (v11);

\draw (v6) -- (v7);
\draw (v6) -- (v9);
\draw (v6) -- (v12);

\draw (v7) -- (v2);

\draw (v8) -- (v9);
\draw (v8) -- (v11);
\draw (v8) -- (v14);

\draw (v9) -- (v10);

\draw (v10) -- (v11);
\draw (v10) -- (v13);

\draw (v11) -- (v0); 

\draw (v12) -- (v13);
\draw (v12) -- (v14);

\draw (v13) -- (v14);

\node at (v0)  {$1$};
\node at (v1)  {$2$};
\node at (v2)  {$3$};
\node at (v3)  {$4$};
\node at (v4)  {$5$};
\node at (v5)  {$6$};
\node at (v6)  {$7$};
\node at (v7)  {$8$};
\node at (v8)  {$9$};
\node at (v9)  {$10$};
\node at (v10) {$11$};
\node at (v11) {$12$};
\node at (v12) {$13$};
\node at (v13) {$14$};
\node at (v14) {$15$};

\end{tikzpicture}

\smallskip
$(4,2,-2)$-graph
\end{minipage}
\hfill
\begin{minipage}{0.5\textwidth}
\centering
\begin{tikzpicture}[
    scale=1,
    every node/.style={circle, draw, fill=white, inner sep=1.2pt, minimum size=3.5mm, font=\tiny,line width=0.7pt},
    thickedge/.style={line width=1pt},
    thinedge/.style={line width=0.45pt}
]
\node (1)  at (0.0,4.2) {$1$};
\node (2)  at (-0.7,2.8) {$2$};
\node (3)  at (0.0,2.8) {$3$};
\node (4)  at (0.7,2.8) {$4$};
\node (5)  at (0.0,1.6) {$5$};
\node (6)  at (0.0,0.7) {$6$};
\node (7)  at (-0.7,-0.8) {$7$};
\node (8)  at (0.7,-0.8) {$8$};

\node (9)  at (3.2,4.2) {$9$};
\node (10) at (2.5,2.8) {$10$};
\node (11) at (3.2,2.8) {$11$};
\node (12) at (3.9,2.8) {$12$};
\node (13) at (3.2,1.6) {$13$};
\node (14) at (3.2,0.7) {$14$};
\node (15) at (2.5,-0.8) {$15$};
\node (16) at (3.9,-0.8) {$16$};

\node (17) at (6.4,4.2) {$17$};
\node (18) at (5.7,2.8) {$18$};
\node (19) at (6.4,2.8) {$19$};
\node (20) at (7.1,2.8) {$20$};
\node (21) at (6.4,1.6) {$21$};
\node (22) at (6.4,0.7) {$22$};
\node (23) at (5.7,-0.8) {$23$};
\node (24) at (7.1,-0.8) {$24$};

\draw[thickedge] (1)--(2) (1)--(3) (1)--(4);
\draw[thickedge] (2)--(5) (3)--(5) (4)--(5);
\draw[thickedge] (2)--(7) (4)--(8);
\draw[thickedge] (6)--(7) (6)--(8);
\draw[thickedge] (7)--(8);

\draw[thickedge] (9)--(10) (9)--(11) (9)--(12);
\draw[thickedge] (10)--(13) (11)--(13) (12)--(13);
\draw[thickedge] (10)--(15)  (12)--(16);
\draw[thickedge] (14)--(15) (14)--(16);
\draw[thickedge] (15)--(16);

\draw[thickedge] (17)--(18) (17)--(19) (17)--(20);
\draw[thickedge] (18)--(21) (19)--(21) (20)--(21);
\draw[thickedge] (18)--(23) (20)--(24);
\draw[thickedge] (22)--(23) (22)--(24);
\draw[thickedge] (23)--(24);

\draw[thinedge] (1)--(9);
\draw[thinedge] (1) to[bend left=8] (17);

\draw[thinedge] (5)--(13) (5)to[bend left=8](21);
\draw[thick] (3) to[bend right=20] (6);
\draw[thick] (11) to[bend right=20] (14);\draw[thick] (19) to[bend right=20] (22);

\draw[thinedge] (2)--(14);
\draw[thinedge] (2)--(24);
\draw[thinedge] (3)--(16);
\draw[thinedge] (3)--(23);
\draw[thinedge] (4)--(15);
\draw[thinedge] (4)--(22);
\draw[thinedge] (6)--(11);
\draw[thinedge] (7)--(19);
\draw[thinedge] (7)--(10);
\draw[thinedge] (8)--(12);
\draw[thinedge] (8)--(18);

\draw[thinedge] (9)--(21);
\draw[thinedge] (10)--(23);
\draw[thinedge] (11)--(22);
\draw[thinedge] (13)--(17);
\draw[thinedge] (14)--(18);
\draw[thinedge] (15)--(4);
\draw[thinedge] (16)--(19);

\draw[thinedge] (6)to[bend right=7](20);
\draw[thinedge] (12)--(24);
\draw[thinedge] (20)--(15);
\end{tikzpicture}

\smallskip
$(5,2,-2)$-graph
\end{minipage}
\captionof{figure}{}\label{P6.fig2}
\end{center}

\begin{theorem}\label{P6.th4}
Let $G$ be the unique $(4,2,-2)$-graph of order $15$. Then $\mu(G) = 9$.
\end{theorem}
\begin{proof}
Let $S\subseteq V(G)$ be a mutual-visibility set of size $s$. By Theorem~\ref{P6.th1} with $d=4$ and $n=15$, we have $s^{2}+7s-180\le 0$, and hence $s\le 10$. Suppose that there exists a mutual-visibility set $S$ with $s=10$, and let $T=V(G)\setminus S$. Then $|T|=5$. By Theorem~\ref{P6.coro6}, $s^{2}+7s-180+e(S,T)\le 0$, and substituting $s=10$ gives $e(S,T)\le 10$. Thus $\sum_{w\in T}|N(w)\cap S|=e(S,T)\le 10$. Since $G$ is $4$-regular, we have $|N(w)\cap S|\in\{0,1,2,3,4\}$ for each $w\in T$, and the function $\binom{x}{2}$ is increasing on this set. Therefore $\sum_{w\in T}\binom{|N(w)\cap S|}{2}$ is maximized for a choice of values $|N(w)\cap S|$ given by $4,4,2,0,0$, and hence
\[
\sum_{w\in T}\binom{|N(w)\cap S|}{2}\le 2\binom{4}{2}+\binom{2}{2}=13.
\]
On the other hand, Lemma~\ref{P6.lem7} gives $2e(S)+e(S,T)=4|S|=40$, and hence $e(S)=20-\frac{e(S,T)}{2}$. By Lemma~\ref{P6.lem3}(i),
\[
45-\left(20-\frac{e(S,T)}{2}\right)\le \sum_{w\in T}\binom{|N(w)\cap S|}{2} \implies 25+\frac{e(S,T)}{2}\le \sum_{w\in T}\binom{|N(w)\cap S|}{2}
\]
In particular, $25\le \sum_{w\in T}\binom{|N(w)\cap S|}{2}$, a contradiction. Hence $s\neq 10$, and therefore $\mu(G)\le 9$. To see that this bound is attained, consider the set
$S=\{2,4,6,8,10,12,13,14,15\}$, where the vertices are labeled as in
Figure~\ref{P6.fig2}. Then every non-adjacent pair of vertices in $S$ has a common neighbour in $V(G)\setminus S$, and hence, by Lemma~\ref{P6.lem1}, $S$ is a mutual-visibility set. Thus $\mu(G)=9$.
\end{proof}
\begin{theorem}\label{P6.th5}
Let $G$ be the unique $(5,2,-2)$-graph of order $24$. Then $\mu(G) \leq 15$.
\end{theorem}
\begin{proof}
Let $S\subseteq V(G)$ be a mutual-visibility set of size $s$. By
Theorem~\ref{P6.th1} with $d=5$ and $n=24$, we obtain $ s^{2}+14s-480\leq 0$, and hence $
s\leq 16$. We next rule out the case $s=16$. Let $ T=V(G)\setminus S$.
Then $|T|=8$. By Theorem~\ref{P6.coro6},
$s^{2}+14s-480+e(S,T)\le 0.$ Substituting $s=16$, we get $256+224-480+e(S,T)\le 0$, 
and therefore $
e(S,T)\le 0$. This is impossible, since $G$ is connected and both $S$ and $T$ are nonempty.
Hence $s\neq 16$, and so
$\mu(G)\le 15$.
\end{proof}
\begin{remark}\label{P6.rem2}
A brief computer search shows that $\{1,5,6,7,8,9,13,14,15,16,22,23,24\}$ is a mutual-visibility set (see Figure~\ref{P6.fig2}), and $\mu(G)=13$.
\end{remark}
\section{Mutual-visibility in Moore graphs}\label{P6.sec4}
We now determine the mutual-visibility number of Moore graphs of diameter $2$, that is, the case of defect zero. Apart from the cycle $C_{5}$, there are only two known Moore graphs of diameter $2$, namely the Petersen graph and the Hoffman--Singleton graph. The visibility polynomial of $C_n$ was determined in~\cite{VP_1}. Since the Petersen graph is geodetic, mutual-visibility sets coincide with general position sets in it. Moreover, the general position polynomial of the Kneser graphs $K(n,2)$ was determined in~\cite{GP_2}, and the Petersen graph is isomorphic to $K(5,2)$. Thus the case of the Petersen graph is already settled, and it remains to study the Hoffman--Singleton graph.
\begin{defn}
The Hoffman-Singleton graph is the unique graph $H$ having the following equivalent descriptions:
\begin{enumerate}
  \item[(i)] $H$ is the unique strongly regular graph with parameters $\rm{srg}(50,7,0,1)$. 
  \item[(ii)] $H$ is the Moore graph of degree $7$ and diameter $2$.
\end{enumerate}
\end{defn}

\begin{prop}\label{P6.lem8}
Let $H$ be the Hoffman-Singleton graph. If $S\subseteq V(H)$ is a mutual-visibility set of size $s$, then $\mu(H)\le 26$.
\end{prop}
\begin{proof}
Let $S\subseteq V(H)$ be a mutual-visibility set, and denote its cardinality by $s$.  
Set $T = V(H)\setminus S$.  From Lemma~\ref{P6.lem7} it follows that
\[
e(S,T)=\sum_{v\in S}\deg_H(v)-2e(S)=ds-2e(S),
\]
where $d=7$ and $n=50$ for the Hoffman-Singleton graph. 
Since each vertex of $T$ has at most $d$ neighbours in $S$,  $
e(S,T)\le d|T|=d(n-s)$. 
Combining the inequality with last equation we get, $ds-2e(S)\le d(n-s)$. Hence,
\[
ds \le \frac{dn}{2} + e(S).
\]
By Corollary~\ref{P6.coro2}, we have, $e(S)\le s/2$. Hence, $
ds \le \frac{dn}{2} + \frac{s}{2}$,
and therefore $(2d-1)s \leq dn$. Thus $
s \le \frac{dn}{2d-1}$,
and for $d=7$, $n=50$ we obtain $
s \le \left\lfloor\frac{7\cdot 50}{13}\right\rfloor = 26$.
Hence $\mu(H)\le 26$, as claimed.
\end{proof}
In Proposition~\ref{P6.lem8}, we establish an upper bound of 26 for the size of any mutual-visibility set in the Hoffman–Singleton graph. Although this provides a useful restriction on the possible size of such sets, the bound is not tight. In the subsequent analysis, we refine our approach and derive a sharper upper bound, showing that the size of a mutual-visibility set in Hoffman-Singleton graph cannot exceed 20.
\begin{theorem}[{\cite[Th.~7.1]{Cvetkovski2012}} ]
Let $f:(a,b)\to\mathbb{R}$ and suppose that $f''(x)$ exists for every $x \in (a,b)$. Then $f$ is convex on $(a,b)$ if and only if $f''(x) \;\geq\; 0 \quad \text{for all } x \in (a,b)$. Moreover, if $f''(x) > 0$ for all $x \in (a,b)$, then $f$ is strictly convex on $(a,b)$.
\end{theorem}

\begin{theorem}\label{th_20}
Let $H$ be the Hoffman-Singleton graph. Then $\mu(H) \leq 20$. Moreover, if a mutual-visibility set $S \subseteq V(H)$ of size $20$ exists, then the induced subgraph $H[S]$ is an induced matching consisting of ten disjoint edges and every vertex of $T = V(H) \setminus S$ has exactly four neighbours in $S$.
\end{theorem}
\begin{proof}
Let $S\subseteq V(H)$ be a mutual-visibility subset of cardinality $s$, and put $T=V(H)\setminus S$.  Let $e(S)$ denote the number of edges in the induced subgraph $H[S]$. 
For each vertex $t$ belonging to $T$, define $k_t$ as the number of neighbours of $t$ lying in $S$, that is, $k_t = |N(t)\cap S|$.  
Then, by Corollary~\ref{P6.coro3} (the case $\delta = 0$)
\begin{equation}\label{P6.eq1}
\binom{s}{2}-e(S)=\sum_{t\in T}\binom{k_t}{2}.
\end{equation}

By definition, let $k_t$ denote the number of edges joining a vertex $t \in T$ to vertices of $S$. Hence $\sum_{t \in T} k_t$ is the total number of edges with one endpoint in $T$ and the other in $S$; 
that is, $
\sum_{t \in T} k_t \;=\; e(S,T)$. 
Therefore, by Lemma~\ref{P6.lem7}, we obtain
\begin{equation}\label{p6.eq4}
\sum_{t \in T} k_t  =  e(S,T)  = 7s - 2e(S)    
\end{equation}
Now let $f(x) = \binom{x}{2}$ for $x \in \mathbb{R}$. Since $f''(x) = 1 > 0$, the function $f$ is strictly convex on $\mathbb{R}$. Applying Jensen’s inequality 
to $f$ with the nonnegative numbers $k_t \geq 0$ (for $t \in T$) and equal weights 
$1/|T|$, we obtain

\begin{align*}
f \left( \frac{1}{|T|}\sum_{t\in T} k_t\right) &\leq \frac{1}{|T|}\sum_{t\in T} f(k_t)\\
f \left( \frac{1}{|T|}(7s-2e(S))\right) & \leq \frac{1}{|T|} \sum_{t\in T}\binom{k_t}{2}\\
\sum_{t\in T}\binom{k_t}{2} &\geq |T|\binom{\frac{7s-2e(S)}{|T|}}{2}
=\frac{(7s-2e(S))^2}{2|T|}-\frac{7s-2e(S)}{2}
\end{align*}
Using the equation \eqref{P6.eq1}, we get
\begin{equation}\label{P6.eq2}
\binom{s}{2}-e(S) \geq \frac{(7s-2e(S))^2}{2(50-s)}-\frac{7s-2e(S)}{2}
\end{equation}
Fix $s$ and, for brevity, write $e = e(S)$. Define
\[
F(e)=\frac{(7s-2e)^2}{2(50-s)}-\frac{7s-2e}{2},
\qquad D(e)=\binom{s}{2}-e-F(e).
\]

The value of \(e\) associated to a mutual-visibility set \(S\) must satisfy \(D(e)\ge 0\). Since $H$ is a Moore graph of diameter 2, by Corollary~\ref{P6.coro2}, $e(S)\le s/2$. Thus, for each fixed $s$, admissible values of $e \in [0, s/2]$ exist precisely when 
$\max_{0 \le e \le s/2} D(e) \ge 0$.
Indeed, if the maximum of \(D\) on \([0,s/2]\) were negative, then \(D(e)<0\) for every feasible \(e\) and no set \(S\) of size \(s\) could satisfy the inequality \eqref{P6.eq2}.

Thus, to obtain a necessary condition on $s$, we examine the maximum value of $D(e)$ over the closed interval $[0, s/2]$. 
A straightforward calculus check shows that $F''(e) = \frac{4}{50 - s}$.
By Proposition~\ref{P6.lem8}, we have $s \le 26$, and hence $F''(e) > 0$. 
Therefore, $F$ is convex, and consequently, $-F$ is concave.
 Since, $\binom{s}{2}-e$ is a linear function in $e$, it is concave. Therefore, $D(e)=\binom{s}{2}-e-F(e)$ is concave on \([0,s/2]\). Consequently, \(\max_{0\le e\le s/2} D(e)\) is attained either at a stationary point or at one of the endpoints. Computing \(D'(e)\) yields the unique critical point \(e_0=4s-25\), and
$$4s-25 > \frac{s}{2} \iff s \geq 8$$
 Hence if $s \geq 8$, then the stationary point lies to the right of the feasible interval and the concave function \(D\) is increasing on \([0,s/2]\); therefore the maximum is attained at the right endpoint \(e=s/2\). Accordingly a necessary condition for existence of a feasible \(e\) is $D \bigl(s/2\bigr)\geq 0$, that is,
\begin{align*}
 \binom{s}{2}-\frac{s}{2}
 &\geq 
\frac{\bigl(7s-2\cdot\frac{s}{2}\bigr)^2}{2(50-s)}-\frac{7s-2\cdot\frac{s}{2}}{2}\\
\frac{s(s-2)}{2} &\ge \frac{(6s)^2}{2(50-s)}-\frac{6s}{2}\\
s-2 &\geq \frac{36s}{50-s}-6\\
(50-s)(s+4) &\geq 36s\quad\Longleftrightarrow\quad s^2-10s-200\le 0
\end{align*}
Hence $s\le 20$. Therefore $\mu(H)\le 20$.

If $s = 20$, then equality holds throughout the above chain. 
In particular, $e(S) = s/2 = 10$, and equality in Jensen’s inequality implies that $k_t$ is constant for all $t \in T$, since $f$ is strictly convex. From Equation~\eqref{p6.eq4}, we have 
\[
\sum_{t \in T} k_t = 7s - 2e(S)
\quad \Longrightarrow \quad 
|T|\,k_t = 7s - 2e(S)
\quad \Longrightarrow \quad 
k_t = \frac{7s - 2e(S)}{|T|} = \frac{7 \cdot 20 - 20}{30} = 4,
\]
for each vertex $t$ belonging to $T$.

Since $s=20$ and $e(S)=10$, it follows from Corollary~\ref{P6.coro1} that $H[S]$ consists of exactly ten disjoint edges, that is, an induced matching. Moreover, since $k_t=4$, each vertex $t\in T$ has exactly four neighbours in $S$.
\end{proof}
Next, we verify that the upper bound $s \le 20$ is sharp. To this end, we formulate the following integer program.

\subsection*{Integer linear program formulation}

Under the hypothesis of Corollary~\ref{P6.coro1}, a subset $S \subseteq V(G)$ is a mutual-visibility set if and only if the induced subgraph $G[S]$ has maximum degree at most $1$. 
This combinatorial condition can be captured by the following set of zero-one constraints.

Let $S$ be a mutual-visibility set of $G$. 
Introduce binary variables $x_v \in \{0,1\}$ for each $v \in V(G)$, where $x_v = 1$ indicates that $v \in S$, and $x_v = 0$ otherwise. 
If $v \in S$, then by Corollary~\ref{P6.coro1}(ii), $\deg_{G[S]}(v) \le 1$; that is, at most one vertex $u \in N_G(v)$ belongs to $S$. 
Hence,
\[
\sum_{u \in N_G(v)} x_u \le 1.
\]
If $v \notin S$, then Corollary~\ref{P6.coro1}(ii) imposes no restriction on $\sum_{u \in N_G(v)} x_u$, which in this case equals $\deg_G(v)$. 
Combining these two cases yields
\[
\sum_{u \in N_G(v)} x_u \;\le\; 1 + \deg_G(v)\,(1 - x_v), 
\qquad \text{for every } v \in V(G).
\]
Thus, the feasible $0$--$1$ vectors $x$ correspond precisely to those subsets $S$ for which $\Delta(G[S]) \le 1$, and therefore, by Corollary~\ref{P6.coro1}, to the mutual-visibility sets of $G$.

Therefore the mutual-visibility number \(\mu(G)\) is given by the following zero-one integer program:
\[
\text{(IP--MV)} \quad \left\{
\begin{aligned}
\text{maximize}\quad & \sum_{v\in V(G)} x_v\\[4pt]
\text{subject to}\quad & \sum_{u\in N_G(v)} x_u \le 1 + \deg_G(v)\,(1-x_v)
\quad \forall v\in V(G),\\[4pt]
& x_v\in\{0,1\}\quad \forall v\in V(G)
\end{aligned}
\right.
\]

For the Hoffman--Singleton graph, the constraints of the integer program (IP-MV) becomes $$\sum_{u\in N_G(v)} x_u \le 1 + 7(1-x_v) \iff \sum_{u\in N_G(v)} x_u + 7x_v \le 8
\quad\text{for all }v\in V(H).$$
The integer linear programming implementation in \texttt{Python} yields an optimal value of \textnormal{(IP--MV)} equal to $20$ for the Hoffman--Singleton graph, with a corresponding mutual-visibility set
\[
S = \{0, 2, 12, 17, 19, 21, 24, 26, 31, 32, 33, 35, 37, 39, 40, 43, 44, 45, 48, 49\}
\]
of size~$20$ where, the vertices are labelled from 0 to 49 and the adjacency matrix of the Hoffman--Singleton graph was taken 
from the function \texttt{hoffman\_singleton\_graph()} in the \texttt{networkx} package.

From Corollary~\ref{P6.coro1}, if the graph $G$ is Moore graph of diameter two then every mutual-visibility set induces a disjoint union of an induced matching and isolated vertices. The integer programming computation shows that the maximum size of a mutual-visibility set in the Hoffman–Singleton graph is 
20. This yields the following result on the size of the largest induced matching in Hoffman-Singleton graph.
\begin{coro}
The maximum size of an induced matching in the Hoffman--Singleton graph is ten.
\end{coro}
\begin{proof}
Note that the Hoffman-Singleton graph $H$ satisfies the hypothesis of Corollary~\ref{P6.coro1}. 
Therefore, by Corollary~\ref{P6.coro1}(iii), a vertex set $S$ is a mutual-visibility set of $H$ if and only if the induced subgraph $H[S]$ is a disjoint union of an induced matching and isolated vertices. Hence, for every induced matching in $H$, the set of its vertices forms a mutual-visibility set. Consequently, the cardinality of the vertex set of any induced matching is at most the maximum cardinality of a mutual-visibility set. Since $\mu(H) \le 20$, no induced matching in $H$ can contain more than $20$ vertices.

By the integer programming formulation, we established that the maximum size of a mutual-visibility set in the Hoffman-Singleton graph is $20$. Furthermore, by Theorem~\ref{th_20}, if $S$ is a mutual-visibility set of the Hoffman-Singleton graph of size $20$, then the induced subgraph $H[S]$ is a disjoint union of ten edges; that is, $H[S]$ forms an induced matching of size~$10$. Hence, the maximum size of an induced matching in the Hoffman-Singleton graph is~$10$.

\end{proof}
\section{Conclusion}
In this paper, we investigated size of mutual-visibility sets in $(d,2,-\delta)$-graphs with $\delta \geq 0$, developing algebraic conditions that govern the mutual-visibility number. 
We then specialized our study to graphs with defect $2$, determining the mutual-visibility number for three of the four known cases and deriving an upper bound for the remaining one. In the defect-zero setting, corresponding to Moore graphs of diameter $2$, we consider Hoffman-Singleton graph. For this graph, we established an upper bound of $20$ for the mutual-visibility number and confirmed its sharpness via an integer programming approach. As a consequence, we obtained that the maximum size of an induced matching in the Hoffman-Singleton graph is $10$. Overall, our results demonstrate that algebraic techniques, when combined with optimization, can yield exact values and sharp bounds for visibility parameters in structured graph classes. This approach suggests several directions for further research,  extending the analysis to graphs of larger diameter or other extremal families.
\section*{Declarations}
\subsection*{Funding sources}
This research did not receive any specific grant from funding agencies in the public, commercial, or not-for-profit sectors.
\subsection*{Conflict of interest}
 The authors declare that they have no conflict of interest.

 \subsection*{Data Availability}No data were used to support this study.

\bibliographystyle{plainurl}
\bibliography{cas-refs}

\end{document}